\documentclass[a4paper,11pt]{amsart} 
\usepackage{cite}
\usepackage{etoolbox}

\makeatletter
\patchcmd{\@setaddresses}{\indent}{\noindent}{}{}
\patchcmd{\@setaddresses}{\indent}{\noindent}{}{}
\patchcmd{\@setaddresses}{\indent}{\noindent}{}{}
\patchcmd{\@setaddresses}{\indent}{\noindent}{}{}
\makeatother

\usepackage[english]{babel}
\usepackage{dsfont} 
\usepackage{graphicx}
\usepackage{color}
\usepackage{enumerate}
\usepackage{array,multirow}

\allowdisplaybreaks

\usepackage{amssymb} 
\usepackage{amsthm} 
\usepackage{dsfont}
\usepackage{amsmath,esint}

\numberwithin{equation}{section}

\theoremstyle{plain}
\newtheorem{theorem}{Theorem}[section]
\newtheorem{lemma}[theorem]{Lemma}
\newtheorem{prop}[theorem]{Proposition}

\newtheorem{remark}[theorem]{Remark}

\theoremstyle{definition}
\newtheorem{df}[theorem]{Definition}
\newtheorem{example}[theorem]{Example}

\begin{document}

\title{Moment problems on compacts of characters of an unital commutative algebra
}

\author{Dragu Atanasiu}

\address{Faculty of Textiles, Engineering and Business\newline University of Bor\aa s\newline All\'egatan 1, 503 32, Bor\aa s\newline  Sweden}

\email{dragu.atanasiu@hb.se}


\maketitle 
\begin{abstract}
In this note  we consider linear functionals on an unital commutative $\mathbb R$-algebra.

We  give an integral representation of a nonnegative functional on an Archimedean  cone where we do not assume that this cone is  a semiring or a quadratic module.

We also give a solution of the moment problem on a product of intervals and determine  conditions for a functional to be a moment functional on a compact of characters.

\end{abstract}

\section{Introduction and preliminaries}

Let $A$ be an unital commutative $\mathbb R$-algebra.
Let $X(A)$ be the set of all $\mathbb R$ algebras homomorphisms from $A$ to $\mathbb R$.
We assume that $X(A)$ is non-empty, and we endow $X(A)$ with the topology of pointwise convergence.
A linear functional $L:A\to \mathbb R$ 
is positive semidefinite if $L(a^2)\ge 0$ for all $a\in A$.

Let $K$ be a compact of $X(A)$. We say that the functional $L:A\to \mathbb R$ is a moment function on $K$ if there is a positive Radon measure $\mu $ on $K$
such that
$$L(a)=\int_{K}\alpha(a)d\mu(\alpha),a\in A.$$
The measure $\mu$ is called the representing measure for $L$.

A cone is a subset $C\subseteq A$ such that 
$C+C\subseteq C$ and $\lambda\cdot C\subseteq C$  for all $\lambda\ge 0$.

A cone $C$ with  $1\in C$ is called an \textit{Archimedean } cone if for every $a\in A$ there is a $\lambda_a\in (0,\infty)$ such that such that $\lambda_a\pm a\in A$.

A quadratic module is a subset $Q\subseteq A$ such that 
$Q+Q\subseteq Q$ and $a^ 2Q\subseteq Q$  for all $a\in A$.

A semiring is a subset $S\subseteq A$ such that 
$S+S\subseteq S, S\cdot S\subseteq S$ and $\lambda\in S$ for all $\lambda\ge 0$.

The function $v:A\to [0,\infty)$ is a submultiplicative function if $v(ab)\le v(a)v(b)$ for all $a,b\in A$ and $v(1)=1$.

The functional $ L:A\to \mathbb R$ is $v-$bounded, where $v$ is a submultiplicative function, if there is a number $C>0$ such that 
$|L(a)|\le Cv(a),a\in A$.

We say that a function $ L:\mathbb N_0=\{0,1,\dots\}\to \mathbb R$ is positive semidefinite on $\mathbb N_0$ if, for every choice of $m\in\mathbb N,n_1,\dots,n_m\in \mathbb N_0$
and $c_1,\dots,c_m\in \mathbb R$ we have
$$\sum_{j,k=1}^ mc_jc_kf(n_j+n_k)\ge 0.$$

A function  $f:\mathbb N_0\to \mathbb R$ is a  moment function  on a compact set  $K\subset \mathbb R$ if there is a positive Radon measure $\mu$ on $K$ such that
$$f(n)=\int_{K}x^ nd\mu(x),n\in \mathbb N_0.$$

In section 2 of this note we construct a submultiplicative function on the algebra $A$.

The motivation of construction in section 2 is the main result in section 3 where we consider an Archimedean cone which a priori  is neither a semiring nor a quadratic module and show, using the multiplicative function from Section 2, that a functional which is nonnegative  on this cone is a moment functional. 

In  section 4 we prove a  Positivstellensatz related to the cone from section 3.

In the next section  we use the submultiplicative function from section 2 to give a solution of the moment problem on a product of intervals and to show that a functional $L:A\to \mathbb R$ is a moment  on a compact set of characters of $A$ if and only if for all $a\in A$ the function $\mathbb N_0\to \mathbb R$ defined by $n\to L(a^ n)$ is a  moment problem on a compact of $\mathbb R$.

In section 6 we consider algebras with a finite number of generators and using that every compact in a finite  dimensional space is included in a ball we establish 
a particularly simple characterization of moment functionals on compacts.

At the end of this paper we compare the submultiplicative function constructed in section 2 with the submultiplicative seminorm in \cite{MSTP}.

The main results of this note are Theorems  \ref{IWOTA13}, \ref{IWOTAINTRINSIC} and  \ref{sf}.

\section{The function $v_L$}
\begin{df}Let $A$ be an unital commutative  $\mathbb R$-algebra, $L:A\to \mathbb R$ a positive semidefinite functional,with $L(1)=1$, and $a\in A$. 
We denote 
$$v_L(a)=\sqrt{\sup_{n\in \mathbb N_0,L(a^{2n})\ne 0}\frac {L(a^{2n+2})}{L(a^{2n})}}$$
when $\sup_{n\in \mathbb N_0,L(a^{2n})\ne 0}\frac {L(a^{2n+2})}{L(a^{2n})}<\infty$
and  $v_L(a)=\infty $ when 
$$\sup_{n\in \mathbb N_0,L(a^{2n})\ne 0}\frac {L(a^{2n+2})}{L(a^{2n})}=\infty.$$
\end{df}
\begin{lemma}\label{arxiv}Let $A$ be an unital commutative  $\mathbb R$-algebra, $L:A\to \mathbb R$ a positive semidefinite functional,with $L(1)=1$, and $a\in A$. We have
$$v_L(a)<\infty$$ if and only if there is a number $c\in [0,\infty)$
such that
$$|L(a^ n)|\le c^ n,n\in \mathbb N_0.$$

If these conditions are satisfied  we have
$$v_L(a)=\inf \{c\in \mathbb [0,\infty)|L(a^ n)\le c^ n,n\in \mathbb N_0\}.$$
\end{lemma}

\begin{proof}Suppose $v_L(a)<\infty$. We have for $n\in \mathbb N_0$

$$L(a^ {2n} )\le (v_L(a))^ 2L(a^ {2(n-1)} )\le\dots \le (v_L(a))^ {2n}L(1)=(v_L(a))^ {2n} $$
which yields 
$$|L(a^ n)|\le\sqrt { L(a^ {2n} )}\le( v_L(a))^ {n}.$$

Now let $c\in[0,\infty)$. 

Suppose that
$$|L(a^ n)|\le c^ n,n\in \mathbb N_0.$$

Let $n\in \mathbb N_0$. Because the function $m\mapsto L(a^ ma^ {2n})$ is positive semidefinite  on $\mathbb N_0$  we get from \cite{BCR},Chapter 4, Proposition 1.12   
$$|L(a^ ma^ {2n})|\le c^ mL(a^ {2n}), m\in \mathbb N_0.$$

Hence we have
$$\sup_{n\in \mathbb N_0,L(a^{2n})\ne 0}\frac {L(a^{2n+2})}{L(a^{2n})}\le c^ 2$$
which  finishes the proof.

\end{proof}
\begin{prop}\label{av}Let $A$ be an unital commutative  $\mathbb R$-algebra, $L:A\to \mathbb R$ a positive semidefinite functional,with $L(1)=1$ 

We suppose that $v_L(a)<\infty$ for all $a\in A$.

The function $v_L:A\to [0,\infty)$ is a submultiplicative function and the function $L$ is $v_L-$bounded.

\end{prop}
\begin{proof}
Because it is easy to see that  $v_L(1)=1$ and we have shown in the proof of Lemma \ref{arxiv} that $|L(a)|\le v_L(a)$ we only have to show that  for all $a,b\in A$
$$v_L(ab)\le v_L(a)v_L(b).$$
We have, for all $n\in \mathbb N_0$,
$$|L(a^ n)|\le (v_L(a))^ n;$$
$$|L(b^ n)|\le (v_L(b))^ n$$
and
$$|L((ab)^ n)|\le v_L(ab))^ n.$$
We also have, for all $n\in \mathbb N_0$,
$$|L((ab)^ n)|^ 2\le L(a^ {2n})L(b^ {2n})\le (v_L(a))^ {2n}(v_L(b))^ {2n}.$$
That is
$$|L((ab)^ n)|\le (v_L(a)v_L(a))^ n.$$
Consequently Lemma \ref{arxiv} yields 
$$v_L(ab)\le v_L(a)v_L(b).$$
\end{proof}
\section{A functional nonnegative on an Archimedean cone}
The next theorem  is the reason of the construction of the function $v_L$ .
To prove it we need the following proposition.
\begin{prop}\label{PreMR}Let $A$ be an unital commutative $\mathbb R$-algebra, a set  $Q\subseteq A$ and $v:A\to [0,\infty)$ a multiplicative function.

For a functional $ L:A\to \mathbb R$, the following conditions are equivalent:
\begin{enumerate}
\item[(i)]
$L$ is  positive semidefinite , $v$-bounded and for all $q\in Q$ the function $x\mapsto L(qx)$ is positive semidefinite;
\item[(ii)]there is an unique positive Radon measure $\mu$ on 
$$K=\{\alpha \in X(A):|\alpha(a)|\le v(a),a\in A\,\, \text{and}\,\,\alpha(q)\ge 0,q\in Q\}$$
such that
$$L(a)=\int_{K}\alpha (a)d\mu(\alpha),a\in A.$$

\end{enumerate}
\end{prop}
\begin{proof}The implication $(ii)\implies (i)$ is immediate. 

$(i)\implies (ii)$ Because the function $a\mapsto L(ab^2)$ is positive semidefinite and $v-$bounded we have by \cite{BCR},Chapter 4, Proposition 1.12 that 

$|L(ab^2)|\le v(a)L(b^2),a,b\in A$. This yields 
$$L(a^2b^2)=|L(a^2b^2)|\le v(a^2)L(b^2)\le (v(a))^2L(b^2),a,b\in A.$$
Consequently $x\mapsto L((v(a))^2-a^2)x)$ is positive semidefinite for all $a\in A$. 
Now the result is a consequence of \cite {DA2}, Theorem 1.4.
\end{proof}

 \begin{theorem}\label{IWOTA13}Let $A$ be an unital commutative $\mathbb R$-algebra

For a functional $ L:A\to \mathbb R$ with $L(1)=1$, the following conditions are equivalent:
\begin{enumerate}
\item[(i)]for every $a\in A$ there is a number $T_a>0$ such that
$$L((T_a-a)^ p(T_a+a)^ q)\ge 0,p,q\in \mathbb N_0;$$
\item[(ii)]there is an positive Radon measure $\mu$ on the compact  
$$ K=\{\alpha \in X(A):|\alpha(a)|\le T_a,a\in A\}$$
such that
$$L(a)=\int_{K}\alpha (a)d\mu(\alpha),a\in A.$$

\end{enumerate}
\end{theorem}
\begin{proof}Because the implication $(ii)\implies (i)$ is obvious we only have to prove $(i)\implies (ii)$.
Let $a\in A$. From the inequalities 
 $$L((T_a-a)^ p(T_a+a)^ q)\ge 0,p,q\in \mathbb N_0$$
we get from \cite{SCH3}, Lemma 3.1
 $$L(a^ 2)\ge 0.$$
Let $n\in \mathbb N_0$. Using again \cite{SCH3}, Lemma 3.1 $n$ times we get
$$L((T_a^{2}-a^ 2)a^{2n})=T_a^{2}L(a^{2n})-L(a^{2(n+1)})\ge 0,n\in \mathbb N_0.$$

Now the result is a consequence of Proposition \ref{av} and Proposition  \ref{PreMR}
 (where we take $Q=\varnothing$).
\end{proof}
 \begin{example}\label{Boll}Let $A=\mathbb R[x_1,\dots,x_s]$ and
$$B^s:=\{(x_1,\dots,x_s)\in \mathbb R^s:x_1^ 2+\dots+x_s^ 2\le 1\}.$$
 The following conditions are equivalent:
\begin{enumerate}
\item[(i)]
$$L((1-x_1^2-\dots -x_s^2)^ \epsilon(1-x_1)^{m_1}(1+x_1)^{n_1}\dots (1-x_s)^{m_s}(1+x_s)^{n_s})\ge 0$$
where $\epsilon \in \{0,1\}$ and $m_1,n_1,\dots ,m_s,n_s\in \mathbb N_0$.

\item[(ii)]L is a moment functional on $B^s$.

\end{enumerate}
\begin{proof} $(i)\implies (ii).$ From \cite{SCH4}, Lemma 12.9 it results that for every $P\in A$ there is a number $M_P>0$ such that 
$$L((M_P-P)^j (M_P+P)^k)\ge 0,j,k\in \mathbb N_0.$$
Consequently by Theorem \ref {IWOTA13} the functional $L$ is positive semidefinite  and bounded by a submultiplicative function.

In the same way for every $P\in A$ we have
$$L((1-x_1^2-\dots -x_s^2)(M_P-P)^j (M_P+P)^k)\ge 0,j,k\in \mathbb N_0$$
which implies, as in the proof of Theorem \ref {IWOTA13}, that the functional
$$P\mapsto L((1-x_1^2-\dots -x_s^2)P)$$ is positive semidefinite.

Now, because $X(\mathbb R[x_1,\dots,x_s])$ can be identified with $\mathbb R^ s$ and because
$$B^s\subset [-1,1]^ s$$
Proposition \ref{av} and Proposition \ref{PreMR} (where we take $Q=\{1-x_1^2-\dots -x_s^2\}$) yield a positive Radon measure $\mu$ on $B^s$  such that
$$L(P)=\int_{B^s}P(x_1,\dots,x_s)d\mu(x_1,\dots,x_s),P\in A.$$
This finishes the proof because the implication $(ii)\implies (i)$ is immediate.

\end{proof}
\end{example}

\section{A positivstellensatz}

\begin{prop}\label{P}Let  $(T_a)_{a\in A} $ is a family of positive numbers and $\mathcal C$ the unital cone generated by the set $S$
$$\left\{(T_a-a)^j (T_a+a)^k,a\in A,j,k\in \mathbb N_0\right\}.$$
If $\alpha(c)>0$ for every $\alpha\in K$ where 

$$K=\{\alpha \in X(A):|\alpha(a)|\le T_a \text { for all }a\in A\}$$
 then $c\in C$.
\end{prop}
\begin{proof}
Because $T_a+a$ and $T_a-a$  are elements of  $\mathcal C$ the unital cone $\mathcal C$ is Archimedean.
If we suppose that $c\notin C$, then, according to \cite{SCH4},Proposition 12.14 or
\cite{SCH33}, Appendix A,Theorem A.27, there is a linear functional $L:A\to \mathbb R$ positive on $\mathcal C$ such that $L(1)=1$ and $L(c)\le 0$.

Using the proof of  Theorem \ref{IWOTA13} we obtain a measure $\mu$ on $K$ such that
$$L(a)=\int_{K}\alpha (a)d\mu(\alpha),a\in A.$$
Now from $L(1)=1$ and $\alpha (c)>0,\alpha\in K$ we get $L(c)>0$. This is not possible and hence $c\in C$.
\end{proof}
\begin{remark}Note that in  \cite{SCH5} are examples of  positivstellens\"atze, for Archimedean cones, which are neither  semirings nor quadratic modules. 
\end{remark}
 \begin{example}\label{pBoll}With the notations from Example \ref{Boll} we obtain, as in Proposition \ref{P}, from Example \ref{Boll} the following positivestellensats:
 
 If $p\in A$ and $p(x)>0$ for all $x\in B^ s$ then $p$ is a nonnegative combination of terms 
 $$(1-x_1^2-\dots -x_s^2)^ \epsilon(1-x_1)^{m_1}(1+x_1)^{n_1}\dots (1-x_s)^{m_s}(1+x_s)^{n_s}$$
 with $\epsilon \in \{0,1\}$ and $m_1,n_1,\dots ,m_s,n_s\in \mathbb N_0$.
 \begin{remark}
For a different representation for polynomials positive on $B^ s$ see \cite{SCH5}, Example 26.
\end{remark}
\end{example}
\section{A moment problem on a product of intervals}
 \begin{theorem}\label{IWOTA}Let $A$ be an unital commutative $\mathbb R$-algebra and $(b_a)_{a\in A}$ a family of positive real numbers.

For a functional $ L:A\to \mathbb R$, with $L(1)=1$, the following conditions are equivalent:
\begin{enumerate}

\item[(i)]for every $a\in A$ there is  a positive Radon measure $\mu_a$ on $[-b_a,b_a]$ such that
$$L(a^ n)=\int_{[-b_a,b_a]}x^ nd\mu_a(x),n\in \mathbb N_0;$$
\item[(ii)]
$L$ is a positive semidefinite functional and we have
$$\sup_{n\in \mathbb N_0,L(a^{2n})\ne 0}\frac {L(a^{2n+2})}{L(a^{2n})}\le b^ 2_a$$
for all $a\in A$;

\item[(iii)]there is an unique positive Radon measure $\mu$ on the compact set 
$$ K=\{\alpha \in X(A):|\alpha(a)|\le b_a,a\in A\}$$
 such that

$$L(a)=\int_{K}\alpha (a)d\mu(\alpha),a\in A.$$

\end{enumerate}
\end{theorem}

\begin{proof}

$(i)\implies (ii)$
we have  $L(a^ 2)=\int_{-b_a}^{b_a} x^2d\mu_a(x)\ge 0$ and for all $n\in \mathbb N_0$

$$L(a^{2n+2})=\int_{-b_a}^{b_a} x^ {2n+2}d\mu(x)=\int_{-b_a}^{b_a} x^ 2x^ {2n}d\mu(x)$$
$$\le b_a^{2}\int_{-b_a}^{b_a} x^{2n}d\mu(x)=b_a^{2}L(a^{2n}).$$
This finishes the proof of the implication $(i)\implies (ii)$;

$(ii)\implies (iii)$ 
we obtain our result from Proposition \ref{av} and Proposition \ref{PreMR}(where we take $Q=\varnothing$);

 $(iii)\implies (i)$ 
the function $\mathbb N_0\to \mathbb R$ defined by $n\mapsto L(a^n)$ is positive semidefinite and we have
$$|L(a^ n)|\le b_a^ n\mu (K).$$
Now the result is a consequence of \cite{BCR},Chapter 4,Theorem 2.5.
\end{proof}
As a consequence of the proof of Theorem \ref{IWOTA} we give 
a characterization of the moment functionals on compacts of characters which depend only of the given functional. 
 \begin{theorem}\label{IWOTAINTRINSIC}Let $A$ be an unital commutative $\mathbb R$-algebra

For a functional $ L:A\to \mathbb R$ with $L(1)=1$, the following conditions are equivalent:
\begin{enumerate}

\item[(i)]
for all $a\in A$ there is an unique positive Radon measure $\mu_a$ on a compact set 
$$ K_a\subset \mathbb R$$
 such that

$$L(a^ n)=\int_{K_a}x^ nd\mu(x),n\in \mathbb N_0;$$
\item[(ii)]
$L$ is a positive semidefinite functional and we have
$$v_L(a)<\infty$$
for all $a\in A$;

\item[(iii)]there is an unique positive Radon measure $\mu$ on a compact set 
$$ K\subseteq X(A)$$
 such that

$$L(a)=\int_{K}\alpha (a)d\mu(\alpha),a\in A.$$
Moreover 
$$ \text{supp}\, \mu \subseteq\{\alpha \in X(A):|\alpha(a)|\le v_L(a),a\in A\}.$$

\end{enumerate}
\end{theorem}
 \section{Algebras with a finite number of generators }
For the  next result we need the following proposition.
\begin{prop}Let $A$ be an unital commutative algebra and $L:A\to\mathbb R $.
 a positive semidefinite functional  $L:A\to\mathbb R $ with $L(1)=1$.

If $v_L(a)<\infty $ for all $a\in A$ then 
the function $v_L$ is a seminorm. That is 
$$v_L(a+b)\le v_L(a)+v_L(b)$$
for all $a,b\in A$ and
$$v_L(\alpha a)=|\alpha| v_L(a)$$
for  $a\in A$ and $\alpha \in \mathbb R$.
\end{prop}
\begin{proof}
We have for all $a,b\in A$ and $n\in\mathbb N_0$
\begin{equation*}\begin{split}|L((a+b)^ n)|&=|\sum_{k=0}^ n\binom nkL(a^ kb^ {n-k})|\\&\le \sum_{k=0}^ n\binom nk|L(a^ kb^ {n-k})|\\&
\overset{Proposition \,\ref{av}}{\le} \sum_{k=0}^ n\binom nkv_L(a)^ kv_L(b)^ {n-k}\\&=(v_L(a)+v_L(b))^ n.\end{split}\end{equation*}
Consequently Lemma\ref{arxiv} yields
$$v_L(a+b)\le v_L(a)+v_L(b).$$

This finishes the proof because the equality $v_L(\alpha a)=|\alpha| v_L(a)$ is immediate.
\end{proof}
\begin{lemma}\label{compact}
Let $A$ be an unital commutative algebra with generators $a_1,\dots,a_m$.
For a positive semidefinite functional  $L:A\to\mathbb R $ with $L(1)=1$ the following conditions are equivalent :
\begin{enumerate}

\item [(i)] 
$$v_L(a_1^ {2}+\dots+a_m^ {2})<\infty;$$
\item[(ii)] $$v_L(a_j)<\infty$$
for all $j\in\{1,\dots,m\};$
\item [(iii)]
$$v_L(a)<\infty,a.$$
for all $a\in A.$

\end{enumerate}
\end{lemma}
\begin{proof}
$(i)\implies (ii)$ Because $L$ is positive semidefinite, we have 
$$(L(a_j^ {n}))^2\le L(a_j^ {2n})\le L((a_1^ {2}+\dots+a_m^ {2})^ {n})\le (v_L(a_1^ {2}+\dots+a_m^ {2}))^ n$$
for all $j\in\{1,\dots,m\}$ and $n\in \mathbb  N_0$. This yields, using Lemma \ref{arxiv},
$$v_L(a_j)\le \sqrt{v_L(a_1^ {2}+\dots+a_m^ {2})}$$
for all $j\in\{1,\dots,m\}$;

$(ii)\implies (iii)$ This is a consequence of the fact that $v_L$ is a submultiplicative function and a seminorm.

The implication $(iii)\implies (i)$ is trivial.
\end {proof}

Now we give an characterization of moment functional on compacts of characters of  an unital commutative algebra, with a finite number of generators
depending only of the functional.
\begin{theorem}\label{sf}
Let $A$ be an unital commutative algebra with generators $a_1,\dots,a_m$.
Let $L:A\to\mathbb R $ be a  positive semidefinite functional  on $A$ such that $L(1)=1$.
Then there exists a representing Radon measure $\mu$ for $L$ with compact support  if and only if 

$$v_L(a_1^ {2}+\dots+a_m^ {2})<\infty$$

Moreover, in this case 
$$\text{supp}\,\mu\subseteq\left\{\alpha\in X(A):|\alpha(a_1)|^ 2+\dots+|\alpha(a_m)|^ 2\le v_L(a_1^ {2}+\dots+a_m^ {2})\right\}.$$
\end{theorem}
\begin{proof}
The result is a consequence of Theorem \ref{IWOTAINTRINSIC} and  Lemma \ref{compact}.
\end{proof}
 \begin{example}\label{Boll2}Let $A=\mathbb R[x_1,\dots,x_s]$ and
 
$$B^s:=\{(x_1,\dots,x_s)\in \mathbb R^s:x_1^ 2+\dots+x_s^ 2\le 1\}.$$
For a functional $L:A\to\mathbb R$, with $L(1)=1$, the following conditions are equivalent:
\begin{enumerate}
\item[(i)]
$L$ is positive semidefinite and
$$v_L(x_1^ {2}+\dots+x_m^ {2})\le 1;$$

\item[(ii)]L is a moment functional on $B^s$.

\end{enumerate}
\begin{proof}

Because $X(\mathbb R[x_1,\dots,x_s])$ can be identified with $\mathbb R^ s$,
Theorem \ref{sf} yields a positive Radon measure $\mu$ on $B^s$  such that
$$L(P)=\int_{B^s}P(x_1,\dots,x_s)d\mu(x_1,\dots,x_s).$$

\end{proof}
\end{example}
\section{More on the function $v_L$}
In the next result we compare the  seminorm $v_L$ with the submultiplicative seminorm
 from \cite {MSTP}.
\begin{theorem}\label{IWOTA33}Let $A$ be an unital commutative $\mathbb R$-algebra and $a\in A$.

For a positive semidefinite functional $ L:A\to \mathbb R$ with $L(1)=1$, the following conditions are equivalent:
\begin{enumerate}
\item[(i)]
$$\sup_{n\in \mathbb N_0,L(a^{2n})\ne 0}\frac {L(a^{2n+2})}{L(a^{2n})}<\infty;$$
\item[(ii)]
$$\sup_{n\in \mathbb N_0,L(a^{2n})\ne 0}\sqrt[2n]{L(a^{2n})}<\infty.$$

Moreover if these inequalities are true   we have 
$$\sup_{n\in \mathbb N}\sqrt[2n]{L(a^{2n})}=\sqrt{\sup_{n\in \mathbb N_0,L(a^{2n})\ne 0}\frac {L(a^{2n+2})}{L(a^{2n})}}.$$
\end{enumerate}
\end{theorem}
\begin{proof}Let $\sup_{n\in \mathbb N_0,L(a^{2n})\ne 0}\frac {L(a^{2n+2})}{L(a^{2n})}<\infty$.

We have, as in Proposition \ref{arxiv}, for all $n\in\mathbb N_0$

$$L(a^ {2n} )\le (v_L(a))^ {2n} $$

Hence 
$$\sqrt[2n]{L(a^{2n})}\le v_L(a).$$
which proves the implication $(i)\implies (ii)$.

Now let $\sup_{n\in \mathbb N}\sqrt[2n]{L(a^{2n})}<\infty$. 

Because the function $a\to \sup_{n\in \mathbb N}\sqrt[2n]{L(a^{2n})}=u(a)$ is a submultiplicative function  (\cite{MSTP},Lemma 3.3), and we have

$$L(a^ n)\le u(a^n)\le u(a)^ n,n\in \mathbb N_0$$

it results from Proposition \ref{arxiv} that 
$$\sup_{n\in \mathbb N_0,L(a^{2n})\ne 0}\frac {L(a^{2n+2})}{L(a^{2n})}<\infty$$
and 
$$v_L(a)\le u(a)$$
which finishes the proof.
\end{proof}

\textbf{Acknowledgments.}

The author thanks Professor Konrad Schm\"udgen for providing information on the reference \cite{SCH5}  where are established  Positivstellens\"atze for  Archimedean cones, which are in general neither  semirings nor  quadratic modules.

\textbf{Data availability declaration}

I do not analyze or generate any datasets because my work proceeds within a theoretical and mathematical approach. One can obtain the relevant materials from the references below.

\textbf{Ethics declaration}

The author declares no competing interests.
\bibliographystyle{amsplain}
\bibliography{references}

\end{document}